\documentclass[11pt]{amsart}
\usepackage{amsmath,amsthm,amssymb,amsfonts,amscd}
\usepackage{amsxtra}
\usepackage{eucal}
\usepackage{mathrsfs}
\usepackage{color, graphics}
\usepackage[all]{xy}
\usepackage{hyperref}
\usepackage{cases}
\headsep=1cm \numberwithin{equation}{section}
\SelectTips{cm}{11}

\def\P{{\mathbb P}}
\def\Z{{\mathbb Z}}

\newcommand \rrlap[1]{\hbox to 0pt{#1}}

\theoremstyle{theorem} 
\newtheorem{Thm}{Theorem}[section]
\newtheorem{Prop}[Thm]{Proposition}
\newtheorem{Coro}[Thm]{Corollary}
\newtheorem{Lem}[Thm]{Lemma}

\theoremstyle{definition}

\newtheorem{Ex}[Thm]{Example}

\newtheorem{Remk}[Thm]{Remark}

\numberwithin{equation}{section}

\begin{document}

\title[Classification of secant defective manifolds]{Classification of secant defective manifolds near the extremal case}
\author[K.\ Han]{Kangjin Han}
\address{Algebraic Structure and its Applications Research Center (ASARC), Department of Mathematics, Korea Advanced Institute of Science and Technology,
373-1 Gusung-dong, Yusung-Gu, Daejeon, Korea}
\email{han.kangjin@kaist.ac.kr}

\thanks{This work was supported by the National Research Foundation of Korea (NRF) grant funded by the Korea
government (MEST) (No. 2011-0001182)}
\date{\today}

\begin{abstract}
Let $X\subset \P^N$ be a nondegenerate irreducible closed subvariety of dimension $n$ over the field of complex numbers and let $SX\subset\P^N$ be its secant variety. $X\subset\P^N$ is called `secant defective' if $\dim(SX)$ is strictly less than the expected dimension $2n+1$. In \cite{Z1}, F.L. Zak showed that for a secant defective manifold necessarily $N\le{n+2 \choose n}-1$ and that the Veronese variety $v_2(\P^n)$ is the only boundary case. Recently R. Mu$\tilde{\textrm{n}}$oz, J. C. Sierra, and L. E. Sol\'a Conde classified secant defective varieties next to this extremal case in \cite{MSS}.

In this paper, we will consider secant defective manifolds $X\subset\P^N$ of dimension $n$ with $N={n+2 \choose n}-1-\epsilon$ for $\epsilon\ge0$. First, we will prove that $X$ is a $LQEL$-manifold of type $\delta=1$ for $\epsilon\le n-2$ (see Theorem \ref{main_thm}) by showing that the tangential behavior of $X$ is good enough to apply Scorza lemma. Then we will completely describe the above manifolds by using the classification of conic-connected manifolds given in \cite{IR1}. Our method generalizes previous results in \cite{Z1,MSS}.
\end{abstract}

\maketitle
\setcounter{page}{1}

\section*{Introduction}\label{section_0}

Let $X\subset \P^N$ be a nondegenerate irreducible closed subvariety of dimension $n$ over the field of complex numbers. If $X$ is also smooth, we call it a \textit{manifold}. Let $SX\subset\P^N$ be the \textit{secant variety} of $X$, i.e. the closure of the union of all lines which pass through two or more points of $X$. 

Clearly, we know that $\dim(SX)$ is less than or equal to $\min\{N,2n+1\}$. If $\dim(SX)<2n+1$, then $X\subset\P^N$ is said to be \textit{secant defective} and we call $\delta(X):=2n+1-\dim(SX)>0$ the \textit{secant defect} of $X$. 

It is well-known that there are no secant defective curves, if $N\ge3$. For surfaces, in \cite{Sev} F. Severi proved that the second Veronese embedding $v_2(\P^2)$ in $\P^5$ and cones are the only secant defective surfaces in $\P^N, N\ge5$. Later G. Scorza classified all secant defective irreducible varieties of dimension $3$ in \cite{Sco}, a result rediscovered by T. Fujita in the smooth case in \cite{F} (see also \cite{CC1} for a modern revisitation of Scorza's original paper \cite{Sco}). In dimension $4$, some partial results have been known (e.g. see \cite{Sco1} and \cite{FR}) but a complete classification in arbitrary dimension seems to be out of reach till now.

One could expect that \textit{the secant variety $SX$ becomes larger as $N$ increases}. In \cite{Z1} Zak proved an upper bound for the embedding dimension of secant defective manifolds:

\begin{Thm}[Zak]
Let $X\subset \P^N$ be a nondegenerate manifold of dimension $n\ge2$ with $\dim(SX)\le2n$ and let $M(n):={n+2 \choose 2}-1$. Then, $N\le M(n)$ with equality holding if and only if $X$ is the second Veronese embedding $v_2(\P^n)\subset\P^{M(n)}$.
\end{Thm}

Therefore, for secant defective manifolds $X\subset \P^N$ it is meaningful to investigate manifolds whose embedding dimension is near this upper bound for $N$, i.e. $M(n)$. Let $B^n\subset\P^{M(n)-1}$ be the projection of $v_2(\P^n)\subset\P^{M(n)}$ from a point on it. Using tangential projections, R. Mu\~noz, J. C. Sierra, and L. E. Sol\'a Conde gave a first anwser in \cite{MSS} as follows:

\begin{Thm}[Mu\~noz, Sierra, Sol\'a Conde]\label{MSC}
Let $X\subset \P^N$ be a nondegenerate secant defective manifold of dimension $n$ and let $N\ge\max\{M(n)-1,2n+1\}$. Then one of the following conditions holds:
\begin{itemize}
\item[(a)] $n\ge2$ and $X$ is the second Veronese embedding $v_2(\P^n)\subset\P^{M(n)}$;
\item[(b)] either $n\ge3$ and $X$ is the isomorphic projection of $v_2(\P^n)\subset\P^{M(n)-1}$, or $n\ge3$ and $X$ is $B^n\subset\P^{M(n)-1}$.
\end{itemize}
\end{Thm}

Here we will consider the secant defective manifolds $X\subset\P^N$ of dimension $n$ with $N=M(n)-\epsilon$ for some $\epsilon\ge0$. We will show first that $X$ is \textit{LQEL}-mainfold for relatively small $\epsilon\ge 0$ (see Theorem \ref{main_thm}) by investigating a general tangential projection via the second fundamental form and deducing that the general tangential behavior of $X$ is good enough to apply Scorza lemma. Next, we give a classification of secant defective manifolds for the cases $\epsilon\le n-2$ (see Corollary \ref{class_almost}) using the classification of conic-connected manifolds in \cite{IR1}, due to P. Ionescu and F. Russo.
\bigskip


\section{Preliminaries}\label{section_1}

We work over the field of complex numbers. Let $X\subset \P^N$ be a nondegenerate irreducible closed submanifold of dimension $n$. 

Let $x,y\in X$ be two general points and $p$ be a general point on the line $\langle x,y\rangle$. We call the closure of the locus of couples of distinct points of $X$ spanning secant lines passing through $p$ \textit{the entry locus of $X$ with respect to $p\in SX$} and denote it by $\Sigma_p(X)$. For a general $p\in SX$, it is easy to see that $\Sigma_p(X)$ is equidimensional of dimension $\delta(X)$. We would like to recall that for two given general points $x,y\in X$ there exists an $r$-dimensional quadric hypersurface $Q_{x,y}^r\subseteq X$ for some $r\ge0$ which passes through $x,y$. Clearly, $Q_{x,y}^r\subseteq\Sigma_p(X)$ when $p\in\langle x,y\rangle$ and $r\le \delta$. Manifolds for which $r\ge1$ are called \textit{conic-connected manifold} (abbr. CCM) and at the other extreme case (i.e. $r=\delta$) they are called \textit{manifold with local quadratic entry locus} (abbr. LQELM). Moreover, when $Q_{x,y}^r=\Sigma_p(X)$ we call $X$ \textit{manifold with quadratic entry locus} (abbr. QELM); see \cite{Ru1, IR1}.

There is also a classical tool, called \textit{Terracini lemma}, which is essential to understand secant varieties (see the classical reference \cite{Ter} and the most recent \cite{Ad}).

\begin{Thm}[Terracini lemma]\label{terracini}
Let $X\subset\P^N_{K}$ be an irreducible subvariety. Then,
\begin{itemize}
\item[(1)] for every $x,y\in X$ $(x\neq y)$, and for every $p\in\langle x,y\rangle$,
$$\langle T_xX,T_yX\rangle \subseteq T_p SX ~;$$
\item[(2)] if $char(K)=0$, there exists an open subset $U$ of $SX$ such that
$$\langle T_xX,T_yX\rangle = T_p SX $$
for every $p\in U$ and $x,y\in X ~(x\neq y)$ such that $p\in \langle x,y\rangle$.
\end{itemize}
\end{Thm}

Tangential projections have been basic tools in the study of secant defective manifolds and of the invariant $\delta(X)$. The \textit{tangential projection} of $X\subset\P^N$, say $\pi_x$, is the projection of $X$ from a general (projective) tangent space to $X$ at a general point $x\in X$, indicated by $T_xX$. Let $W_x\subset\P^{N-n-1}$ be the Zariski closure of the image of $\pi_x$ and let $\delta=\delta(X)$. Then, we have

\begin{equation}\label{tg_pr}
  \pi_x : X\dashrightarrow W_x \subset \P^{N-n-1}~,
\end{equation}
where $\dim(W_x)=n-\delta$ by Terracini Lemma (Theorem \ref{terracini}). In particular the general fiber of $\pi_x$ is  purely $\delta$-dimensional. 

When $X\subset\P^N$ is secant-defective and $SX\subsetneq\P^N$, we can describe the image of the general tangential projection, $W_x\subset \P^{N-n-1}$, via the second fundamental form of $X$ at $x$. Let $\phi: Bl_xX\to X$ be the blow-up of $X$ at $x$, $E=\P((T_xX)^{\ast})=\P^{n-1}\subset Bl_xX$ be the exceptional divisor, and $H$ be a hyperplane section of $X\subset\P^N$. Consider the rational map $\tilde{\pi}_x : Bl_xX\dashrightarrow \P^{N-n-1}$ and its restriction to $E$ which is given by the linear system $|\phi^{\ast}(H)-2E|_{|E}\subseteq |-2E_{|E}|=|\mathcal O_{\P^{n-1}}(2)|$. We call this linear system the \textit{second fundamental form of $X\subset\P^N$ at $x$} and denote it by $|II_{x,X}|$. Obviously, $\dim(|II_{x,X}|)\le N-n-1$ and $\tilde{\pi}_x(E)\subseteq W_x\subseteq \P^{N-n-1}$. 

Let $\mathcal{L}_x$ be the Hilbert scheme of lines contained in $X$ and passing through $x\in X$. The scheme $\mathcal L_x$ can be naturally identified with a subscheme of $E\simeq\P^{n-1}$ parametrizing the space of tangent diriections to $X$ at $x$. Then
\begin{equation}\label{tg_pr}
  \mathcal{L}_x \subseteq Bs(|II_{x,X}|)~(\textrm{the base locus of $|II_{x,X}|$})~,
\end{equation}
since the base locus of $|II_{x,X}|$ consists of directions associated to lines having a contact at least three with $X$ at $x$. 

We recall here a key result which allows us to see the \textit{whole} general tangential projection $W_x$ via the second fundamental form $|II_{x,X}|$ (see \cite{IR2} Proposition 1.2).

\begin{Prop}\label{IR}
Let $X$ be a nondegenerate secant defective manifold and let $x$ be a general point of $X$. Then $\dim(|II_{x,X}|)=N-n-1$ and $|II_{x,X}|$ is surjective (i.e. $\tilde{\pi}_x(E)=W_x$).
\end{Prop}

It is also remarkable that \textit{good} tangential behavior guarantees \textit{simple} entry loci of secant defective manifolds as the following classical result of Scorza assures (see, for example, \cite{Sco,Ru1}):

\begin{Thm}[Scorza Lemma]\label{scorza} Let $X\subset\P^N$ be an irreducible nondegenerate $n$-dimensional variety of secant defect $\delta(X)=\delta\ge1$ such that $SX\subsetneq \P^N$. Suppose that a general tangential projection $\pi_x(X)=W_x\subset\P^{N-n-1}$ is an irreducible variety having birational Gauss map. Let $y\in X$ be a general point. Then
\begin{itemize}
\item[(a)] the irreducible component of the closure of fiber of the rational map $\pi_x:X\dasharrow W_x\subset\P^{N-n-1}$ passing through $y$ is either an irreducible quadric hypersurface of dimension $\delta$ or a linear space of dimension $\delta$, the last case occuring only for singular varieties.
\item[(b)] There exists on $X\subset\P^N$ a $2(n-\delta)$-dimensional family $\mathcal Q$ of quadric hypersurfaces of dimension $\delta$ such that through two general points $x,y\in X$ there passes a unique quadric $Q_{x,y}$ of the familiy $\mathcal Q$. Furthermore, the quadric $Q_{x,y}$ is smooth at $x,y$ and it consists of the irreducible components of $\Sigma_p(X)$ passing through $x$ and $y$, $p\in<x,y>$ general.
\item[(c)] If $X$ is smooth, then a general member of $\mathcal Q$ is smooth.
\end{itemize}
\end{Thm}

In particular, we deduce from part (b) of Theorem \ref{scorza}, that as soon as $W_x$ has birational Gauss map, then $X\subset \P^{N}$ is a LQELM of type $\delta=\delta(X)>0$ and in particular a \textit{CCM}.

\begin{Remk}\label{gen_finite}
Note that it is enough to assume that a general tangential projection $W_x$ has generically finite Gauss map in order to deduce that $X$ is a LQELM via Scorza lemma. In fact, in case of characteristic zero, generically finiteness of the Gauss map $\mathcal G_{W_x}$ is equivalent to that $\mathcal G_{W_x}$ is birational by linearity of general contact loci (see \cite{Z1} Theorem 2.3 (c) , pg. 21) and Zak's Theorem on Tangencies (see \cite{Z1} Corollary 1.8, pg. 18).
\end{Remk}

Recently, a classification of \textit{CCM} has been obtained by Ionescu and Russo in \cite{IR1}. Recall that $X$ is Fano if $-K_X$ is ample. We call $X$ a {\it prime} Fano of index $i(X)$ if $Pic(X)\simeq\Z\langle\mathcal O_X(1)\rangle$ and $-K_X=i(X)H$ for some positive integer $i(X)$, where $H$ is the hyperplane section class of $X\subset \P^{N}$. Since conic-connectedness is stable under isomorphic projection, we may assume that $X\subset \P^{N}$ is linearly normal.

\begin{Thm}[Ionescu, Russo]\label{CC} Let $X\subset\P^N$ be a linearly normal CCM of dimension $n$. Then either $X\subset \P^{N}$ is a prime Fano manifold of index $i(X)\ge \frac{n+1}{2}$, or it is projectively equivalent to one of the following:
\begin{itemize}
\item[(a)] the second Veronese embedding $v_2(\P^n)$ in $\P^{M(n)}$;
\item[(b)] the projection of $v_2(\P^n)$ from the linear space $\langle v_2(\P^s)\rangle$, where $\P^s\subset \P^n$ is a linear subspace such that $N=M(n)-{s+2 \choose 2}$ and $0\le s \le n-2$; equivalently $X\simeq Bl_{\P^s}(\P^n)$ embedded in $\P^N$ by the linear system of quadric hypersurfaces of $\P^n$ passing through $\P^s$.
\item[(c)] the Segre embedding $\P^a\times\P^b\subset\P^{ab+a+b}$, where  $a,b\ge1$ and $a+b=n$.
\item[(d)] A hyperplane section of the Segre embedding $\P^a\times\P^b\subset\P^{N+1}$, where $n\ge3$, $N=ab+a+b-1$, $a,b\ge2$ and $a+b=n+1$.
\end{itemize}
\end{Thm}

\begin{Remk}\label{cov_line}
We would like to mention that every prime Fano $X$ of index $i(X)\ge\frac{n+1}{2}$ in the classification of conic-connected manifolds of Theorem \ref{CC} is actually \textit{covered by lines}, i.e. through each point of $X$ there passes a line contained in $X$. When $i(X)>\frac{n+1}{2}$, this comes from S. Mori's work in \cite{Mo}. In case of $i(X)=\frac{n+1}{2}$, it is one of the consequences of \cite{CMSB} due to K. Cho, Y. Miyaoka and N.I. Shepherd-Barron.
\end{Remk}

\section{Secant defective manifolds near the extremal case}\label{section_2}

Here is our main results.

\begin{Thm}\label{main_thm}
Let $X\subset \P^N$ be a nondegenerate secant defective manifold of dimension $n\ge2$ with $SX\subsetneq \P^N$ and let $N= M(n)-\epsilon$ for some $\epsilon\ge0$. Then:
\begin{itemize}
\item[(i)] if $\epsilon\le n-2$, then $\delta(X)=1$ while for $n-1\le\epsilon$, we have 
$$1\le\delta(X)\le\min\{\epsilon-n+2,\frac{n}{2}\}~.$$
\item[(ii)] if $\epsilon\le n-2$, then $X\subset \P^{N}$ is a LQELM of type $\delta(X)=1$ (in particular, $X$ is a CCM).
\end{itemize}
\end{Thm}

\begin{proof}
Let $\delta=\delta(X)>0$. Zak's Linear Normality Theorem (see, for example, \cite{Z1} Corollary 2.17, pg. 48) assures that under our hypothesis, $\delta\le\frac{n}{2}$. Take a general point $x\in X$ and consider the tangential projection $\pi_x:X\subset \P^N \dasharrow W_x\subset\P^{N-n-1}$. From Proposition \ref{IR}, we have a diagram

\begin{equation}\label{tg_proj}
    \xymatrix @R=3.5pc @C=3.5pc{
 &Bl_x(X)\ar@{-->}[dr]^{\widetilde{\pi_x}}\ar[d]_{\sigma}& E\simeq\P^{n-1}\ar@{-->}[d]^{|II_{x,X}|}\ar@{_{(}->}[l] \ar@{^{(}->}[dr]^{|\mathcal O_{\P^{n-1}}(2)|} &\\
  &**[l]\P^N\supset X\ar@{-->}[r]_{\pi_x } &**[r] W_x^{n-\delta} \subset \P^{N-n-1} &**[r] v_2(\P^{n-1})\subset\P^{M(n-1)}~. \ar@{-->}[l]^-{f}
  }
\end{equation}

Then
\begin{eqnarray*}
\dim|\mathcal O_{\P^{n-1}}(2)| - \dim |II_{x,X}| &=& M(n-1)-\{M(n)-\epsilon-n-1\}\\
&=&\frac{(n-1)(n+2)}{2}-\frac{n(n+3)}{2}+\epsilon+n+1 = \epsilon~,
\end{eqnarray*}
so that $f$ (see the diagram \ref{tg_proj}) should be a projection of the second Veronese $Y:=v_2(\P^{n-1})$ from a linear subspace $\Lambda\simeq\P^{\epsilon-1}$. Moreover $\delta(Y)=1$ and for any two distinct points $y_1,y_2\in Y$ we have $\dim\langle T_{y_1}Y,T_{y_2}Y\rangle = 2n-2$. Let $y_1,y_2\in Y$ be two distinct points such that $f(y_1)$ and $f(y_2)$ are smooth points of $W_x$ and such that
$$ T_{f(y_1)}W_x = T_{f(y_2)}W_x~\textrm{and}~\langle \Lambda, T_{y_i}Y\rangle = \langle \Lambda, T_{f(y_i)} W_x\rangle~\textrm{for each $i=1,2$}.$$

Then $\langle T_{y_1}Y,T_{y_2}Y\rangle\subseteq \langle \Lambda, T_{f(y_i)}W_x \rangle$, yielding
\begin{displaymath}
2n-2\le\dim\langle \Lambda,T_{f(y_i)}W_x\rangle = (n-\delta)+\epsilon \le n-1+\epsilon
\end{displaymath}
which implies $\delta\le\epsilon-n+2$ and $\epsilon \ge n-1$.\\

(i) If $\epsilon\le n-2$, then by the above argument $f$ is birational and $\delta(X)=1$. Suppose now $\epsilon\ge n-1$. If $f$ is birational, then $1=\delta(X)\le\min\{\epsilon-n+2,\frac{n}{2}\}$; if $f$ is not birational, then by the above argument $\delta(X)\le\epsilon-n+2$ and part (i) is completely proved.\\

(ii) Let $y^\prime$ be a general point of $W_x\subset\P^{N-n-1}$ and suppose that $\mathcal G_{W_x}$ is not birational. Then, by linearity of general contact loci and Zak's Theorem on Tangencies (see Remark \ref{gen_finite}), we know that there exists an (affine) line $L\subset W_x$ passing through $y^\prime$ along which the tangent spaces to $W_x$  are constant.

By applying the above argument to points $y_1, y_2\in Y$ lying in two distinct fibers over general points of $L$ we would deduce $\epsilon\ge n-1$, contrary to the hypothesis in (ii). Therefore, part (ii) follows by applying Scorza lemma (see Theorem \ref{scorza} and Remark \ref{gen_finite}). 

\end{proof}

\begin{Remk}\label{rationality}
There are some remarks on this Theorem \ref{main_thm}.

\begin{itemize}
\item[(a)] In \cite{Z1} Zak proved that for a secant defective manifold $X\subset\P^N$ of dimension $n$ with $\delta=\delta(X)$, necessarily $N\le g([\frac{n}{\delta}],n,\delta)$, see \cite{Z1} for the definition of the function $g$. Then one easily see that $g([\frac{n}{2}],n,2) < M(n) -n+2$ so that the first part of (i) in the above Theorem is also a consequence of more general results, which could also provide a proof of (ii) above. The above proof is elementary and simple and uses slightly different techniques.
\item[(b)] It is obvious that one can not expect that any secant defective manifold $X\subset\P^{M(n)-\epsilon}$ is a CCM for \textit{all} $\epsilon\ge0$. Let us consider an example of some infinite family of $4$-dimensional secant defective manifolds, which is well-known (see also \cite{FR}). Let $V=\P^2\times\P^2$ be a Segre 4-fold in $\P^8$ which is a secant defective manifold with $\delta(V)=2$. Choose a point $p\in\P^9 \setminus \P^8$ and consider the cone $Z=C_p(V)\subset\P^9$. Since $SZ=C_p(SV)$, the secant variety $SZ$ is 8-dimensional. Take the intersection of $Z$ and $H$, a \textit{general} hypersurface of degree $d\ge2$ not passing through $p$. Denote this 4-dimensional intersection by $X$. Then, $X\subset\P^9$ becomes smooth, $SX=SZ=C_p(SV)$, $\dim(SX)=8$ and $K_X \thicksim (d-4)H$. Therefore, for a large $d>0$, $X$ is to be \textit{of general type}.
\end{itemize}
\end{Remk}

Finally, we are going to present a classification of $n$-dimensional secant defective manifolds $X\subset\P^N$ \textit{near the extremal case}, that is to say the cases of $N=M(n)-\epsilon$ for $\epsilon\le n-2$, which extends the previous results of \cite{Z1,MSS}.

\begin{Coro}\label{class_almost}
Let $X\subset \P^N$ be a nondegenerate secant defective manifold of dimension $n\ge 2$ with $SX\subsetneq \P^N$ and let $N\ge M(n)-(n-2)$. Then $X$ is projectively equivalent to one of the following:
\begin{itemize}
\item[(i)] the second Veronese embedding $v_2(\P^n)$ in $\P^{M(n)}$;
\item[(ii)] the isomorphic projection of $v_2(\P^n)$ into $\P^{M(n)-\epsilon}$ for any $1\le\epsilon\le n-2$;
\item[(iii)] the projection of $v_2(\P^n)$ from the linear space $\langle v_2(\P^s)\rangle$ (we call this $B^n_s$), where $\P^s\subset\P^n$ is any linear subspace such that $s\ge 0, {s+2\choose 2}\le n-2$;
\item[(iv)] the isomorphic projection of $B^n_s\subset\P^{M(n)-{s+2 \choose 2}}$ into $\P^{M(n)-\epsilon}$ for any $\epsilon$ with ${s+2 \choose 2}<\epsilon\le n-2$.
\end{itemize}
\end{Coro}

\begin{proof}
By Theorem \ref{main_thm}, any secant defective manifold $X\subset\P^N$ is a LQELM for $N=M(n)-\epsilon$, where  $0\le\epsilon\le n-2$ and $n\ge 2$. In particular $X$ is one of the list in the \textit{CCM} classification of Theorem \ref{CC}. It is easy to see that the second Veronese $v_2(\P^2)\subset\P^5$ is the only case for $n=2$. From now on, let us assume that $n\ge3$.\\

First, we claim that $X\subset\P^N$ is not prime Fano in our range. Recall that $\mathcal{L}_x\subset E$ is the Hilbert scheme of lines contained in $X$ and passing through the general point $x\in X$, which under this hypothesis is equidimensional. From (\ref{tg_pr}) we have $\mathcal L_{x}\subseteq Bs(|II_{x,X}|)=Bs(f)$ (the base locus). Since $\delta(X)=1$, Theorem 2.3 in \cite{Ru} yields $\dim\mathcal{L}_x =\frac{n-3}{2}\ge 0$. Note that the dimension of the linear span of $v_2(\mathcal{L}_x)$ is greater than or equal to ${\frac{n-3}{2}+2 \choose 2}-1$ with equality holding if and only if $\mathcal{L}_x\subset E$ is a linear subspace, a case which under the prime Fano hypothesis is excluded by Theorem 1.1 in \cite{Ar}. Furthermore, we have $\langle v_2(\mathcal{L}_x)\rangle\subseteq\langle Bs(f)\rangle\subseteq\P^{\epsilon-1}$ so that 
\begin{equation}\label{a1}
{\frac{n-3}{2}+2 \choose 2}-1 < \dim\langle v_2(\mathcal{L}_x)\rangle\le\epsilon-1\le n-3~,
\end{equation}
which is impossible.

Since $\delta(\P^a\times\P^b)=2$, we have only to exclude case (d) in Theorem \ref{CC}. In this case

\begin{displaymath}
M(n)-\epsilon\ge\frac{n(n+3)}{2}-(n-2) > \frac{(n+3)}{2}\frac{(n+3)}{2}-2\ge (a+1)(b+1)-2
\end{displaymath}
where $a\ge2, b\ge2, a+b=n+1$, and $n\ge3$.\\

This completes the proof.
\end{proof}

With the help of Corollary \ref{class_almost}, for instance, we could classify all secant defective 5-folds whose embedding dimension is near the upper bound $M(5)=20$ as follows:

\begin{Ex}[Secant defective 5-folds near the extremal case]\label{new_defect}
Let $X$ be a secant defective 5-fold in $\P^N$ with $N\ge M(5) - 3$ (i.e. $N=17, 18, 19, 20$). Then, $X$ is one of the following:
\begin{itemize}
\item[(i)] the second Veronese embedding $v_2(\P^5)$ in $\P^{20}$;
\item[(ii)] either the isomorphic projection of $v_2(\P^5)$ into $\P^{19}$ or $B^5_0$ in $\P^{19}$;
\item[(iii)] either the isomorphic projection of $v_2(\P^5)$ into $\P^{18}$ or the isomorphic projection of $B^5_0$ into $\P^{18}$;
\item[(iv)] either the isomorphic projection of $v_2(\P^5)$ into $\P^{17}$, the isomorphic projection of $B^5_0$ into $\P^{17}$ or $B^5_1$ in $\P^{17}$.
\end{itemize}

\end{Ex}

{\bf Acknowledgements}  The author would like to thank Professor Paltin Ionescu for having introduced him to this problem during PRAGMATIC 2010, for encouragement and for valuable suggestions. He is also grateful to the organizers of PRAGMATIC 2010 for providing a good chance to participate and to Dr. Jos\'e C. Sierra for his useful comments on this work. Finally, he thanks the referee for many corrections and helpful suggestions making this paper more readable.

\end{document}